\documentclass[11pt,twoside,a4paper]{article}

\author{Shai Sarussi }

 \title{Quasi-valuations - topology and the weak approximation theorem}
 \date{}

\usepackage{amsmath,amsthm,amssymb}

\begin{document}

\newtheorem{thm}{Theorem}[section]
\newtheorem{cor}[thm]{Corollary}
\newtheorem{lem}[thm]{Lemma}
\newtheorem{prop}[thm]{Proposition}
\newtheorem{ax}{Axiom}

\theoremstyle{definition}
\newtheorem{defn}[thm]{Definition}

\theoremstyle{remark}
\newtheorem{rem}[thm]{Remark}
\newtheorem{ex}[thm]{Example}
\newtheorem*{notation}{Notation}

\newcommand{\qv}{{quasi-valuation\ }}


\maketitle

\begin{abstract} {Suppose $F$ is a field with a nontrivial valuation $v$
and valuation ring $O_{v}$, $E$ is a finite field extension and
$w$ is a quasi-valuation on $E$ extending $v$. We study
the topology induced by $w$. We prove that
the quasi-valuation ring determines the topology, independent of the choice of its
quasi-valuation. Moreover, we prove the weak approximation
theorem for quasi-valuations.

}\end{abstract} 

\section{Introduction}

\ \ Recall that a valuation on a field $F$ is a function $v : F
\rightarrow \Gamma \cup \{  \infty \}$, where $\Gamma$ is a
totally ordered abelian group and $v$ satisfies the following
conditions:

(A1) $v(x) = \infty$ iff $x=0$;

(A2) $v(xy) = v(x)+v(y)$ for all $x,y \in F$;

(A3) $v(x+y) \geq \min \{ v(x),v(y) \}$ for all $x,y \in F$.

There has been considerable interest in recent years in generalizations of valuations, in order to treat rings that are not integral domains, and also to handle several valuations simultaneously. For example, pseudo-valuations
({\it see} [Co],[Hu], and [MH]), Manis-valuations and PM-valuations ({\it see}
[KZ]), value functions ({\it see} [Mor]), and gauges ({\it see} [TW]). These related theories are
discussed briefly in the introduction of [Sa].

In this paper we continue our study from [Sa] of quasi-valuations.
Recall that a {\it quasi-valuation} on a ring $A$ is a function $w
: A \rightarrow ~M \cup \{  \infty \}$, where $M$ is a totally
ordered abelian monoid, to which we adjoin an element $\infty$
greater than all elements of $M$, and $w$ satisfies the following
properties:

(B1) $w(0) = \infty$;

(B2) $w(xy) \geq w(x) + w(y)$ for all $x,y \in A$;

(B3) $w(x+y) \geq \min \{ w(x), w(y)\}$ for all $x,y \in A$.

The minimum of a finite number of valuations with the same value
group is a quasi-valuation. For example, the $n$-adic
quasi-valuation on $\Bbb Q$ (for any positive $n \in \Bbb Z$)
already has been studied in [Ste]. (Stein calls it the $n$-adic
valuation.) It is defined as follows: for any $0 \neq \frac{c}{d}
\in \Bbb Q$ there exists a unique $e \in \Bbb Z$ and integers $a,b
\in \Bbb Z$, with $b$ positive, such that
$\frac{c}{d}=n^{e}\frac{a}{b}$ with $n \nmid a$, $(n,b)=1$ and
$(a,b)=1$. Define $w_{n}(\frac{c}{d})=e$ and $w_{n}(0)=\infty$.

In [Sa] we develop the theory of quasi-valuations on finite dimensional field extensions that extend a
given valuation. For the reader's convenience we briefly overview some of the results from [Sa].
Let $F$ be a field with valuation $v$
and valuation ring $O_{v}$, let $E$ be a finite field extension and let
$w$ be a quasi-valuation on $E$ extending $v$ with a corresponding quasi-valuation ring $O_{w}$.
We prove that $O_{w}$ satisfies INC (incomparability), LO (lying over),
and GD (going down) over $O_{v}$; in particular, $O_{w}$ and $O_{v}$ have the
same Krull Dimension. We also prove that every such
quasi-valuation is dominated by some valuation extending $v$. Namely, there exists a valuation $u$
extending $v$ on $E$ so that $\forall x \in E$,
$w(x) \leq u(x)$.

Under the assumption that the value monoid of the quasi-valuation
is a group we prove that $O_{w}$ satisfies GU (going
up) over $O_{v}$, and a bound on the size of the prime spectrum is
given. In addition, a 1:1 correspondence is obtained between
exponential quasi-valuations and integrally closed quasi-valuation
rings.

Given $R$, an algebra over $O_{v}$, we construct a quasi-valuation
on $R$; we also construct a quasi-valuation on $R \otimes _{O_{v}}
F$, which helps us prove our main Theorem. The main Theorem states
that if $R \subseteq E$ satisfies $R \cap F=O_{v}$ and $E$ is the
field of fractions of $R$, then $R$ and $v$ induce a
quasi-valuation $w$ on $E$ such that $R=O_{w}$ and $w$ extends
$v$; thus $R$ satisfies the properties of a quasi-valuation ring.

In this paper, we extend some fundamental results from valuation theory. For example, we prove that
the topology induced by a quasi-valuation is Hausdorff and
totally disconnected. We also prove a weak version of the approximation theorem for quasi-valuations.











\section{The topology induced by a quasi-valuation}

In this section we introduce the topology induced by a
quasi-valuation. We show that this topology is close to the
topology induced by a valuation in the sense that they share
some basic topological properties such as being both Hausdorff and
totally disconnected. In the main theorem of this section we prove that the
topology induced by the quasi-valuation is determined by the corresponding quasi-valuation ring.

In this section $F$ denotes a field with a nontrivial valuation
$v$, a value group $\Gamma$, and a valuation ring $O_{v}$.
$E$ denotes a finite dimensional field extension with $n=[E:F]$,
$w:E \rightarrow M \cup \{ \infty \}$ is a quasi-valuation on $E$
with quasi-valuation ring $O_{w}$ (namely, $O_{w}=\{x\in E \ | \ w(x)\geq0\}$),
such that $w|_{F}=v$, and $M$ is a totally ordered abelian
monoid containing $\Gamma$.

We note that w(-1)=0 because $w$ extends $v$. Thus, by [Sa, Lemma 1.3], we have $w(x)=w(-x)$ for all $x \in E$.
Moreover, by [Sa, Lemma 1.4], for all $x,y \in E$ such that $w(x) \neq w(y)$, we have $w(x+y)= min \{w(x),w(y)\}$. Recall from [Sa, Definition 1.5] that an element $c \in E$ is called stable
with respect to $w$ if $w(cx)=w(c)+w(x)$ for every $x \in E$. Thus, by [Sa, Lemma 1.6], every $a\in F$ is stable with respect
to $w$.
We shall freely use these facts throughout the paper.

Let $x \in E$ and $m \in M$; we denote $$U_{m}^{w}(x)=\{ y \in E \mid
w(y-x)>m \};$$
we suppress $w$ when
it is understood. 

\begin{rem} \label{x in umx} Let $x \in E$ and $m \in M$; then $x \in U_{m}(x)$. Indeed,
$$w(x-x)=w(0)=\infty>m.$$ \end{rem}

We shall repeatedly use Remark \ref{x in umx} without reference.


\begin{lem} If $y \in U_{m_{1}}(x_{1}) \cap
U_{m_{2}}(x_{2})$ and $m_{1} \leq m_{2}$, then $$
y \in U_{m_{2}}(y) \subseteq U_{m_{1}}(x_{1}) \cap
U_{m_{2}}(x_{2}).$$\end{lem}

\begin{proof} Since $y \in U_{m_{1}}(x_{1}) \cap U_{m_{2}}(x_{2})$,
we have $w(y-x_{1}) > m_{1}$ and \linebreak $w(y-x_{2}) > m_{2}
\geq m_{1}$. Let $z \in U_{m_{2}}(y)$; then $w(z-y) > m_{2}$. Thus,
$$w(z-x_{1}) = w(z-y+y-x_{1})$$ $$ \geq min \{ w(z-y),
w(y-x_{1})\} > m_{1}$$ and $$w(z-x_{2}) = w(z-y+y-x_{2})$$ $$ \geq
min \{ w(z-y), w(y-x_{2})\} > m_{2}.$$ Thus $z \in U_{m_{1}}(x_{1}) \cap
U_{m_{2}}(x_{2})$

\end{proof}

We denote $B=\{ U_{m}(x) \mid x \in E$, \ $m \in M\}$.

\begin{cor} The set $B$ is a base for a topology on $E$.

\end{cor}

In view of Corollary 2.3 we define,

\begin{defn} The topology whose base is $B$ will be denoted by $T_{w}$. We call $T_{w}$ the topology induced by the quasi-valuation $w$.\end{defn}

We recall the following lemma from [Sa, Lemma 2.8]:

\begin{lem} \label{wx alpha} Let $E/F$ be a finite field extension and
let $w$ be a quasi-valuation on $E$ extending a valuation $v$ on
$F$. Then $w(x) \neq \infty$ for all $0\neq x \in E$. In fact, for all $0\neq x \in E$, there exists
$\alpha \in \Gamma$ such that $w(x) < \alpha$. \end{lem}


We denote
$$M^{G}=\{ m \in M \mid \text{\ there exists\ } 0 \neq y\in E
\text{\ such that\ } m \leq w(y)  \}.$$

\begin{rem} \label{MG is a submonoid and contains Gamma} $M^{G}$ is a submonoid of $M$ containing $\Gamma$.

\end{rem}

\begin{proof}
Let $m_1,m_2 \in M^{G}$;
then there exist nonzero $y_1, y_2 \in E$ such that $m_1 \leq w(y_1)$ and $m_2 \leq w(y_2)$. Thus,
$$m_1 +m_2 \leq w(y_1)+w(y_2) \leq w(y_1y_2).$$ Note that $y_1y_2 \neq 0$ and thus $m_1 +m_2 \in M^{G}$. It is easy to see that $\Gamma \subseteq M^{G}$; indeed, for every $\alpha \in \Gamma$ there exists a nonzero $a \in F$ such that $w(a)=v(a)=\alpha$.

\end{proof}

\begin{rem} \label{MG does not have a maximal element} $M^{G}$ does not have a maximal element; in fact, for all $m \in M^G$ there exists
 $\alpha \in \Gamma \cap  M^G$ such that $m < \alpha$. \end{rem}

\begin{proof} Let $m \in M^G$.
Then there exists $0 \neq y \in E$ such that $m \leq w(y)$. By Lemma \ref{wx alpha} there exists
$\alpha \in \Gamma$ such that $w(y) < \alpha$. So, $m < \alpha $. By Remark \ref{MG is a submonoid and contains Gamma}, $\alpha \in \Gamma \subseteq M^G$.
\end{proof}


\begin{prop} \label{T_w is discrete} $T_{w}$ is discrete iff there exists an
element $m \in M \setminus M^{G}$.
\end{prop}

\begin{proof} ($\Leftarrow$) Let $m \in M \setminus M^{G}$ and let $x \in E$. Then for
every $y \neq x$ we
have $w(y-x)<m$; thus $U_{m}(x)=\{ x \}$.


($\Rightarrow$) We assume $M=M^{G}$ and we show that $T_{w}$ is
not discrete. It is enough to show that every open set has
infinitely many elements. Now, since every open set contains some
$U_{m}(x)$ (for $m \in M$, $x \in E$), it is enough to show that
every $U_{m}(x)$ has infinitely many elements. By our assumption
$M=M^{G}$ and thus for every $m \in M$ there exists $0 \neq z
\in E$ such that $m \leq w(z)$; also, by Lemma \ref{wx alpha}, for every such $z$ there
exists $0 < \alpha \in \Gamma$ such that $w(z) < \alpha $. Take
$a \in O_{v}$ with $v(a)=\alpha$. Then $x+a^{n} \in U_{m}(x)$ for
each ~$ n \in \Bbb N$, proving the set $U_{m}(x)$ has infinitely
many elements.



\end{proof}

In view of Proposition  \ref{T_w is discrete}, we restrict our discussion to $M^{G}$;
namely we denote $B=\{ U_{m}(x) \mid x\in E,\ m \in M^{G} \}$ as
a base for $T_{w}$.


\begin{lem} \label{w y-x >m} Let $x,y,z \in E$ and $m \in M^{G}$. If $z \in U_{m}(x) \cap
U_{m}(y)$ then $w(y-x) >m$.\end{lem}

\begin{proof} By definition, $z \in U_{m}(x) \cap
U_{m}(y)$ implies $w(z-x)>m$ and $w(y-z)>m$. Thus,
$$w(y-x) \geq \min \{w(y-z),w(z-x) \}>m.$$

\end{proof}

\begin{prop} $T_{w}$ is Hausdorff.\end{prop}

\begin{proof} Let $x, y \in E$ with $x \neq y$, and write
$w(y-x)=m \in M^{G}$. 
By Lemma \ref{w y-x >m} we have, $$U_{m}(x) \cap U_{m}(y) =
\emptyset .$$

\end{proof}

\begin{lem} \label{U_mx is clopen} Let $x \in E$ and $m \in M^{G}$. Then $U_{m}(x)$ is closed as well as
open.\end{lem}

\begin{proof} Let $y \notin U_{m}(x)$; then $w(y-x) \leq m$.
By Lemma \ref{w y-x >m} we have, $$U_{m}(y)\cap U_{m}(x)= \emptyset;$$
obviously $y \in U_{m}(y)$.
So, $U_{m}(y)$ is an open set containing $y$ disjoint from $U_{m}(x)$.


\end{proof}

The following lemma shows that $E$ is totally disconnected, in the
following sense.

\begin{prop} The only nonempty connected subsets of $E$
are the singleton sets $\{ x \}$ for $x \in E$.\end{prop}

\begin{proof} Let $S \subseteq E$ be a nonempty set containing at
least two elements, $x \neq y$. Write $w(x-y)=m$ and $U_{1}=U_{m}(x)$. Let $U_{2}$ denote the complement of $U_{1}$ in $E$, which is open by
Lemma \ref{U_mx is clopen}. Note that $x \in U_{1}$ and $y \in U_{2}$ (since
$w(y-x)=m \ngtr m$), and thus by definition $S$ is disconnected.

\end{proof}

Let $x\in E$ and $m \in M^{G}$; we denote $\widetilde{U}_{m}(x)=\{ y \in E \mid w(y-x) \geq m \}$.
Obviously, $ U_{m}(x) \subseteq \widetilde{U}_{m}(x)$. Thus, as in Remark \ref{x in umx}, we have $x \in
\widetilde{U}_{m}(x)$. 

\begin{lem} \label{U_my subset of U_mx} Let $x,y \in E$ and $m \in M^G$.
If $y \in \widetilde{U}_{m}(x)$ then $\widetilde{U}_{m}(y) \subseteq \widetilde{U}_{m}(x).$
If $y \notin \widetilde{U}_{m}(x)$ then $ \widetilde{U}_{m}(y) \subseteq (\widetilde{U}_{m}(x))^{c}.$\end{lem}

\begin{proof} Suppose $y \in \widetilde{U}_{m}(x)$; then $w(y-x) \geq m$. Let $z \in \widetilde{U}_{m}(y)$;
then $w(z-y) \geq m$. Hence, $w(z-x) \geq min \{w(z-y), w(y-x) \} \geq m$. Thus, $$
\widetilde{U}_{m}(y) \subseteq \widetilde{U}_{m}(x).$$

Suppose $y \notin \widetilde{U}_{m}(x)$; then $w(y-x) < m$.
Let $z \in \widetilde{U}_{m}(y)$;
then $w(z-y) \geq m$. Hence, $w(z-x) = min \{w(z-y), w(y-x) \} = w(y-x)<m$. Thus,
$$ \widetilde{U}_{m}(y) \subseteq (\widetilde{U}_{m}(x))^{c}.$$

\end{proof}

\begin{cor} \label{widetildeU_mx is clopen} $\widetilde{U}_{m}(x)$ is both open and
closed, for any $x \in E$ and $m \in M^{G}$.\end{cor}

\begin{proof} Let $y \in \widetilde{U}_{m}(x)$; then by Lemma \ref{U_my subset of U_mx},$$y \in
U_{m}(y) \subseteq \widetilde{U}_{m}(y) \subseteq \widetilde{U}_{m}(x).$$ Let $y \notin \widetilde{U}_{m}(x)$; then by Lemma \ref{U_my subset of U_mx},
$$y \in U_{m}(y) \subseteq \widetilde{U}_{m}(y) \subseteq (\widetilde{U}_{m}(x))^{c}.$$

\end{proof}




\begin{lem} \label{formula} Let $x \in E$ and let $m,m' \in
M^{G}$ such that $m<m'$. Then
$$U_{m}(x)=\bigcup_{y \in U_{m}(x)} \widetilde{U}_{m'}(y).$$
\end{lem}

\begin{proof} $(\subseteq)$ holds because every $y  \in U_{m}(x)$ is
obviously in $\widetilde{U}_{m'}(y)$. To prove $(\supseteq)$, we
need to show that $\widetilde{U}_{m'}(y) \subseteq U_{m}(x)$ for
all $y \in U_{m}(x)$. So, let $y \in U_{m}(x)$ and let $z \in \widetilde{U}_{m'}(y)$.
Then, $w(y-x)>m$ and $w(z-y) \geq m'$. Thus, since $m<m'$,
$$w(z-x) \geq min \{ w(z-y),w(y-x)\}>m.$$

\end{proof}

We denote $B_{1}=\{ \widetilde{U}_{m}(x) \mid x\in
E, \ m \in M^{G}\}$.

\begin{prop} The set $B_{1}$ is a base for $T_{w}$.
\end{prop}

\begin{proof} First, by Corollary \ref{widetildeU_mx is clopen}, $\widetilde{U}_{m}(x)$ is open
for all $x \in E$ and $ m \in M^{G}$. Now, let $x \in E$ and $m \in M^{G}$. By Remark \ref{MG does not have a maximal element}
there exists $m<m' \in M^{G}$. By Lemma \ref{formula} we have, $U_{m}(x)=\bigcup_{y \in U_{m}(x)} \widetilde{U}_{m'}(y).$ Thus, every open set in $T_{w}$ is a union of elements of $B_{1}$.
\end{proof}


In fact, we can describe the topology in terms of $\Gamma$ (the
value group of the valuation) as the following proposition shows.

First, we denote $B_{2}=\{ \widetilde{U}_{\alpha}(x) \mid x\in
E, \ \alpha \in \Gamma\}$.

\begin{prop} \label{B_2 is a base for thetopology} The set $B_{2}$ is a base for $T_{w}$. 
\end{prop}

\begin{proof} First, by Remark \ref{MG is a submonoid and contains Gamma}, $B_{2} \subseteq B_{1}$.
So every element of $B_{2}$ is open in $T_{w}$. Now, let $x \in E$ and $m \in M^{G}$.
By Remark \ref{MG does not have a maximal element}, there exists
 $\alpha \in \Gamma \cap  M^G$ such that $m < \alpha$. By Lemma \ref{formula},
$$U_{m}(x)=\bigcup_{y \in U_{m}(x)} \widetilde{U}_{\alpha}(y).$$

\end{proof}

Recall from [Sa, Section 10] that for every ring $R \subseteq E$
satisfying $R \cap F=O_{v}$, we denote $$\mathcal W_{R}=\{ w \mid~
w \text{\ is a quasi-valuation on\ } E \text{\ extending\ } v
\text{\ with\ } O_{w}=R\}.$$


Also recall that the class $\mathcal W_{R}$ is not empty, by [Sa, Theorem 9.35].


\begin{lem} \label{wx geq va iff xa-1>0} Let $w \in \mathcal W_{R}$, $ x \in E$ and $0 \neq a \in F$. The following are equivalent:

(a) $w(x) \geq v(a);$

(b) $w(x)-v(a) \geq 0;$

(c) $w(xa^{-1}) \geq 0;$

(d) $xa^{-1} \in R.$

\end{lem}

\begin{proof} (a)$\Leftrightarrow$(b). Because $v(a) \in \Gamma$. (b)$\Leftrightarrow$(c). $v$ is a valuation and $0 \neq a \in F$; thus $-v(a)=v(a^{-1})$. Therefore, $w(x) - v(a)=w(x) + v(a^{-1})$. Since $w$ extends $v$ and $a$ is stable with respect to $w$, we get
$$w(x) + v(a^{-1})=w(x) + w(a^{-1})=w(xa^{-1}).$$ (c)$\Leftrightarrow$(d). By assumption, $w \in \mathcal W_{R}$; thus $O_{w}=R$.
\end{proof}

\begin{lem} \label{w_1x iff w_2x} Let $w_{1},w_{2} \in \mathcal W_{R}$ and let
$\alpha \in \Gamma$; then $$w_{1}(x)\geq \alpha \text{\ iff\
}w_{2}(x) \geq \alpha.$$
\end{lem}

\begin{proof} By assumption, $O_{w_1}=O_{w_2}=R$. Let $a \in F$ such that $v(a)=\alpha$; clearly,
$a \neq 0$ (since $\alpha \in \Gamma$). Thus, using Lemma \ref{wx geq va iff xa-1>0} twice, we get
$$w_{1}(x)\geq \alpha \text{\ iff\ } xa^{-1} \in O_{w_1}=O_{w_2} \text{\ iff\ } w_{2}(x)\geq \alpha.$$

\end{proof}

We are ready to prove the main theorem of this section.

\begin{thm} If $w_{1},w_{2} \in \mathcal W_{R}$, then
$T_{w_{1}}=T_{w_{2}}$. In other words, the quasi-valuation ring
determines the topology, independent of the choice of its
quasi-valuation.

\end{thm}

\begin{proof} By Proposition \ref{B_2 is a base for thetopology}, the set $C=\{ \widetilde{U}_{\alpha}^{w_{1}}(x) \mid x\in E, \ \alpha \in \Gamma\}$ is a base for $T_{w_{1}}$ and the set $D=\{ \widetilde{U}_{\alpha}^{w_{2}}(x) \mid x\in E, \ \alpha \in \Gamma\}$ is a base for $T_{w_{2}}$. However, for every $x\in
E$ and $ \alpha \in \Gamma$ we have, by Lemma \ref{w_1x iff w_2x}, $\widetilde{U}_{\alpha}^{w_{1}}(x)=\widetilde{U}_{\alpha}^{w_{2}}(x)$. Thus, $C=D$ and the theorem is proved.


\end{proof}


\section{Weak approximation theorem}

In this section we prove a weak version of the approximation
theorem for quasi-valuations. We call it the weak approximation
theorem since it relies on the independence of the valuation rings
in $F$ and not on the independence of the quasi-valuations in $E$.
(The independence of the valuation rings in $F$ implies the
independence of the quasi-valuation rings in $E$ but not vice
versa).

In this section $F$ denotes a field and
$E$ denotes a finite dimensional field extension with $[E:F]=n$.

\begin{defn} Let $A$ and $B$ be two subrings of
$F$. $A$ and $B$ are called independent if $AB=F$. Two valuations
are called independent, if their rings are independent; likewise,
two valuations are called dependent, if their rings are dependent.
\end{defn}

We recall the Approximation Theorem for valuations.

\begin{thm} (Approximation Theorem for valuations)
([Bo, Section 7.2, Thm. 1]) Let $\{ v_{i}\}_{1 \leq i \leq k}$ be
a set of valuations on a field $F$ which are independent in pairs
and let $\Gamma_{i}$ be the value group of $v_{i}$. Let $x_{i} \in
F$ and $\alpha_{i} \in \Gamma_{i}$ for $1 \leq i \leq k$. Then
there exists an $x \in F$ such that $v_{i}(x-x_{i}) \geq \alpha
_{i}$ for all $i$.
\end{thm}

Let $\{ O_{v_{i}} \}_{1 \leq i \leq k}$ be a finite set
of valuation rings of $F$. We denote by $B$ their intersection,
i.e.,
$$B= \bigcap _{1 \leq i \leq k} O_{v_{i}}.$$ Let $\{ R_{i} \}_{1
\leq i \leq k}$ be a finite set of subrings of $E$ such that $E$
is the field of fractions of each $R_{i}$ and $R_{i} \cap F=
O_{v_{i}}$ for every $1 \leq i \leq k$. Recall that by [Sa, Theorem
9.35], for every $1 \leq i \leq k$ there exists a filter
quasi-valuation $w_{i}$ on $E$ corresponding to $R_{i}$ and such
that $w_{i}$ extends $v_{i}$ (so the collections of quasi-valuations
corresponding to the $R_{i}$'s are not empty.) We shall prove our
theorem for every quasi-valuation $w_{i}$ on $E$ corresponding to
$R_{i}$ (not necessarily filter quasi-valuations). Note that for
every $1 \leq i \leq k$, $R_{i}F$ is an integral domain finite
dimensional over $F$ and thus a field containing $R_{i}$; hence
$R_{i}F=E$. Moreover, by [Bo, Section 7, Proposition 1], the field
of fractions of $B$ is $F$.






We denote $R= \bigcap _{1 \leq i \leq k} R_{i}$.





The next observation is well known.

 \begin{rem} Let $C$ be an integral domain, $S$ a
multiplicative closed subset of $C$ with $0 \notin S$, and $R$ an
algebra over $C$. We claim that every $x \in R \otimes_{C}CS^{-1}$
is of the form $r \otimes \frac{1}{\beta}$ for $r \in R$ and
$\beta \in S$. Indeed, write $x=\sum_{i=1}^{t}(r_{i} \otimes
\frac{\alpha_{i}}{\beta_{i}})$ where $r_{i} \in R$, $\alpha_{i}
\in C$ and $\beta_{i} \in S$. Let $\beta=\Pi_{i=1}^{t}\beta_{i}$
and $\alpha_{i}'=\alpha_{i}\beta \beta_{i}^{-1} \in C$. Thus,
$$\sum_{i=1}^{t}(r_{i} \otimes \frac{\alpha_{i}}{\beta_{i}})=
\sum_{i=1}^{t}(r_{i} \otimes \frac{\alpha_{i}'}{\beta})=
\sum_{i=1}^{t} (\alpha_{i}' r_{i} \otimes \frac{1}{\beta})=r
\otimes \frac{1}{\beta}.$$

Where $r=\sum_{i=1}^{t} \alpha_{i}' r_{i}$.\end{rem}

\begin{prop} $E=S^{-1}R \cong R \otimes_{B} F$,
where $S=B \setminus \{ 0\}$.\end{prop}

\begin{proof} $S^{-1}R$ is an integral domain finite dimensional
over $F$, so is a field. It remains to show that any $x \in E$ has
the form $r/b$ where $r\in R$ and $b \in S$. By the previous Remark, $x$ can be written in the form $r_{i}/b_{i}$ where $r_{i}\in
R_{i}$ and $b_{i}\in B \setminus \{ 0\}$. Write $b= \prod_{1 \leq i \leq k} b_{i}$
and get
$$bx=r_{i} \prod _{l \neq i}b_{l} \in R_{i}B=R_{i}$$ for every
$1 \leq i \leq k$, and thus $x=bx/b$ has the desired form.\end{proof}

We are ready for the main theorem of this section: the weak
approximation theorem for quasi-valuations.

\begin{thm} Let $E/F$ be a finite field extension and
let $\{ O_{v_{i}} \}_{1 \leq i \leq k}$ be a finite set of
valuation rings of $F$ which are pairwise independent. Let $\{
R_{i} \}_{1 \leq i \leq k}$ be a finite set of subrings of $E$
such that the field of fractions of each $R_{i}$ is $E$ and $R_{i}
\cap F= O_{v_{i}}$ for every $1 \leq i \leq k$. Let $\{ w_{i} \mid w_i:E \rightarrow M_i \cup \{ \infty\}, \ 1 \leq i \leq k \}$
be a set of quasi-valuations on $E$ such that for every $1 \leq i \leq
k$, $w_{i} \in \mathcal W_{R_i}$. Let $\{ x_{i}
\}_{1 \leq i \leq k} \subseteq E$ and let $\{ m_{i} \}_{1 \leq i
\leq k}$ be a set of elements such that, for every $1 \leq i \leq
k$, $m_{i} \in M^{G}_{i}$. Then there exists an element $x \in E$
such that
$$w_{i}(x-x_{i}) \geq m_{i}.$$ for all $1 \leq i \leq k$.
\end{thm}

\begin{proof} Since $m_{i} \in M^{G}_{i}$ for every $1 \leq i \leq
k$, we get by Remark \ref{MG does not have a maximal element} that there exist $\alpha_{i} \in \Gamma_{i}$ such that $m_{i}
< \alpha_{i}$ for all $1 \leq i \leq
k$. We shall prove that $w_{i}(x-x_{i}) \geq
\alpha_{i}$ for every $1 \leq i \leq k$. Let $R= \bigcap
_{1 \leq i \leq k} R_{i}$; by Proposition 3.4, $R$
contains a basis $\{ r_{1}, r_{2},...,r_{n}\} $ of $E$ over $F$.
Write, for every $1 \leq i \leq k$,
$$x_{i}=\sum_{1 \leq j \leq n} c_{ij}r_{j}$$ where $c_{ij} \in F$.
The approximation theorem for valuations gives $d_{1},...,d_{n}
\in ~F$ such that $$v_{i}(d_{j}-c_{ij}) \geq \alpha_{i}, $$ for $1
\leq i \leq k$, $1 \leq j \leq n$.


Define $x= \sum _{1 \leq j \leq n} d_{j}r_{j}$ and get, for every
$1 \leq i \leq k$,
$$w_{i}(x-x_{i})=w_{i}(\sum _{1 \leq j \leq n} d_{j}r_{j}-\sum_{1
\leq j \leq n} c_{ij}r_{j})$$
$$=w_{i}(\sum _{1
\leq j \leq n} (d_{j}- c_{ij})r_{j}).$$

Note that, for every $1 \leq j \leq n$ and $1 \leq i \leq k$, $$w_{i}(d_{j}-
c_{ij})=v_{i}(d_{j}- c_{ij}) \geq \alpha_{i}$$ and $w_{i}(r_{j})
\geq 0$ (since $r_{j} \in R$). Thus, $$w_{i}(\sum _{1
\leq j \leq n} (d_{j}- c_{ij})r_{j}) \geq \min_{1
\leq j \leq n} w_{i}( (d_{j}- c_{ij})r_{j}) \geq \alpha_i.$$

\end{proof}

Department of Mathematics, Sce College, Ashdod, 77245, Israel.

{\it E-mail address: sarusss1@gmail.com}

\end{document}